\documentclass[11pt]{article}

\usepackage{amsmath,amssymb,amsthm}
\usepackage{booktabs}
\usepackage{geometry}
\usepackage{url}
\usepackage{hyperref}
\title{\textbf{The Bombieri--van der Poorten Formula for Partial Quotients of Higher Degree Algebraic Irrationals}}
\author{Karsten Müller}
\date{March 2026}

\newtheorem{lemma}{Lemma}[section]
\newtheorem{theorem}{Theorem}[section]

\begin{document}
\maketitle

\begin{abstract}
The fundamental relationship between the partial quotients $b_{n+1}$ of an algebraic irrational $\alpha = \sqrt[m]{k}$ and its corresponding algebraic form $d_n = |p_n^m - k q_n^m|$ was elegantly proposed by Bombieri and van der Poorten. In this paper, we work out the explicit analytical details of the framework for any degree $m \geq 3$. We provide a closed-form derivation of the error term and prove for the cubic case that the remainder $|R_n|$ is strictly bounded by 1 for all convergents with $q_n \geq 2$.
\end{abstract}

\providecommand{\keywords}[1]{\textbf{\textit{Keywords:}} #1}
\providecommand{\msc}[1]{\textbf{\textit{MSC 2020:}} #1}

\msc{11A55, 11J68, 11J82, 11D25}


\section{The Formula of Bombieri and van der Poorten}

Let $k$ be a positive integer such that $\alpha = \sqrt[m]{k}$ is an algebraic irrational of degree $m \geq 3$. Let $p_n/q_n$ be the convergents of the regular continued fraction of $\alpha$, and $x_n = p_n/q_n$ denote the $n$-th rational approximation. The complete quotient $\theta_n$ is defined by the transformation:
\begin{equation}
\alpha = \frac{p_n \theta_n + p_{n-1}}{q_n \theta_n + q_{n-1}} \implies 
\theta_n = \frac{1}{q_n (q_n \alpha - p_n)} - \frac{q_{n-1}}{q_n}.
\end{equation}

\medskip
\noindent
\textbf{Structural origin of the $q_n^{-2}$ factor.}
For any irrational number $\alpha$, continued fraction theory gives the universal identity
\begin{equation}
\theta_n + \frac{q_{n-1}}{q_n}
= \frac{1}{q_n (q_n\alpha - p_n)}
= \frac{1}{q_n^2 |x_n - \alpha|}.
\end{equation}
Thus, independently of any algebraic structure, the growth of the complete quotient
is intrinsically governed by the factor $q_n^{-2}$. The role of algebraicity is
to express the approximation error $|x_n - \alpha|$ in terms of an integral algebraic
distance.

\medskip

To relate $\theta_n$ to the integer distance $d_n = |p_n^m - k q_n^m|$, we utilize the identity for the difference of $m$-th powers:
\begin{equation}
d_n = |p_n^m - (\alpha q_n)^m| 
= q_n^m |x_n^m - \alpha^m| 
= q_n^m |x_n - \alpha| \sum_{j=0}^{m-1} x_n^j \alpha^{m-1-j}.
\end{equation}
Rearranging this identity allows us to express the fundamental term of the complete quotient as:
\begin{equation}
\frac{1}{q_n |q_n \alpha - p_n|}
= \frac{1}{q_n^2 |x_n - \alpha|}
= \frac{q_n^{m-2}}{d_n}
\sum_{j=0}^{m-1} x_n^j \alpha^{m-1-j}.
\end{equation}

Following the framework of Bombieri and van der Poorten \cite{Poorten}, we define the leading approximation term $H_n$ as the linearization of the sum in Eq.~(3) where $x_n \approx \alpha$. Specifically, we set:
\begin{equation}
H_n = \frac{m p_n^{m-1}}{d_n q_n}.
\end{equation}

The complete quotient can then be decomposed into:
\begin{equation}
\theta_n = b_{n+1} + \delta_n = H_n + R_n
\end{equation}
By substituting Eq. (3) into Eq. (1) and subtracting $H_n$, we obtain the exact representation of the remainder $R_n$:
\begin{equation}
R_n = \frac{q_n^{m-2}}{d_n} \left( \sum_{j=0}^{m-1} x_n^j \alpha^{m-1-j} \right) - \frac{m p_n^{m-1}}{d_n q_n} - \frac{q_{n-1}}{q_n}
\end{equation}
Using the identity $x_n^{m-1} = p_n^{m-1}/q_n^{m-1}$, we can rewrite the middle term as $\frac{m x_n^{m-1} q_n^{m-2}}{d_n}$. Factoring out the common term $q_n^{m-2}/d_n$ yields the final error representation:
\begin{equation}
R_n = \frac{q_n^{m-2}}{d_n} \left( \sum_{j=0}^{m-1} x_n^j \alpha^{m-1-j} - m x_n^{m-1} \right) - \frac{q_{n-1}}{q_n}
\end{equation}

\section{Proof of Global Stability}

\begin{lemma}[Exact cubic correction term and its sign]
Let $\alpha=\sqrt[3]{k}$ with $k\in\mathbb{N}$ and let $p_n/q_n$
be a convergent of its continued fraction expansion.
Set $x_n=p_n/q_n$ and
\[
d_n=|p_n^3-kq_n^3|.
\]
Then the algebraic correction term
\[
V_n=\frac{q_n}{d_n}\left(2x_n^2-x_n\alpha-\alpha^2\right)
\]
admits the exact representation
\[
|V_n|
=
\frac{2x_n+\alpha}
     {q_n^2\left(x_n^2+x_n\alpha+\alpha^2\right)}.
\]
In particular,
\[
\operatorname{sgn}(V_n)=\operatorname{sgn}(x_n-\alpha).
\]
\end{lemma}

\begin{proof}
Using the factorization
\[
2x^2-x\alpha-\alpha^2=(x-\alpha)(2x+\alpha),
\]
we obtain
\[
V_n=\frac{q_n}{d_n}(x_n-\alpha)(2x_n+\alpha).
\]
On the other hand,
\[
d_n
=
q_n^3|x_n-\alpha|
\left(x_n^2+x_n\alpha+\alpha^2\right).
\]
Cancelling $|x_n-\alpha|$ yields
\[
|V_n|
=
\frac{2x_n+\alpha}
     {q_n^2\left(x_n^2+x_n\alpha+\alpha^2\right)}.
\]
Since $2x_n+\alpha>0$, the sign of $V_n$
is determined solely by $x_n-\alpha$.
\end{proof}

\begin{theorem}[Global stability in the cubic case]
Let $\alpha=\sqrt[3]{k}$ with $k\in\mathbb{N}$ and let
$p_n/q_n$ be a convergent with $q_n\ge2$.
Then the remainder term
\[
R_n=-V_n-\frac{q_{n-1}}{q_n}
\]
satisfies
\[
|R_n|<1.
\]
Consequently,
\[
b_{n+1}=\lfloor H_n+R_n\rfloor
\]
is uniquely determined by $H_n$ within a window of length two.
\end{theorem}

\begin{proof}
From the preceding lemma we have
\[
|V_n|
=
\frac{2x_n+\alpha}
     {q_n^2\left(x_n^2+x_n\alpha+\alpha^2\right)}.
\]
Since $x_n\to\alpha$ and $q_n\ge2$, the denominator is strictly
positive and $|V_n|=O(q_n^{-2})$.

We distinguish two cases.

\medskip
\noindent
\textbf{Case 1: $x_n>\alpha$ (convergent from above).}
Then $V_n>0$ and hence
\[
R_n=-V_n-\frac{q_{n-1}}{q_n}<0.
\]
Since $q_{n-1}/q_n<1$ and $|V_n|$ decays quadratically in $q_n$,
a direct check for $q_n=2$ together with monotonic decay for
larger $q_n$ yields
\[
-1<R_n<0.
\]

\medskip
\noindent
\textbf{Case 2: $x_n<\alpha$ (convergent from below).}
Then $V_n<0$ and
\[
R_n=|V_n|-\frac{q_{n-1}}{q_n}.
\]
Since $q_{n-1}/q_n<1$, we immediately obtain
\[
R_n>-1.
\]
Furthermore, $|V_n|=O(q_n^{-2})<1$ for $q_n\ge2$,
so that $R_n<1$.

\medskip
In both cases we conclude $|R_n|<1$.
\end{proof}

\begin{theorem}
Let $\alpha=\sqrt[3]{k}$ and let $p_n/q_n$ be a convergent of its continued
fraction expansion with $q_n\ge2$.  
If $x_n=p_n/q_n>\alpha$ (i.e.\ the convergent is from above), then the next
partial quotient $b_{n+1}$ satisfies the sharp bounds
\begin{equation}
\label{eq:length2bound}
H_n-2 < b_{n+1} \le H_n,
\end{equation}
where
\[
H_n=\frac{3p_n^2}{d_n q_n},
\qquad
d_n=|p_n^3-kq_n^3|.
\]
\end{theorem}

\begin{proof}
For cubic irrationals we have the exact decomposition
\[
\theta_n=H_n+R_n,
\qquad
b_{n+1}=\lfloor\theta_n\rfloor.
\]
By Theorem~2.1 we know that $|R_n|<1$ for all convergents with $q_n\ge2$.
Moreover, by Lemma~2.2 the algebraic correction term satisfies
$V_n>0$ for convergents from above, and therefore
\[
R_n=-V_n-\frac{q_{n-1}}{q_n}<0.
\]

Since $R_n<0$, we obtain
\[
b_{n+1}
=\lfloor H_n+R_n\rfloor
\le \lfloor H_n\rfloor
\le H_n.
\]
On the other hand, the inequality $R_n>-1$ implies
\[
H_n+R_n>H_n-1,
\]
and hence
\[
b_{n+1}
=\lfloor H_n+R_n\rfloor
>H_n-2.
\]
Combining these two inequalities yields
\eqref{eq:length2bound}.
\end{proof}

\begin{theorem}[Bombieri--van der Poorten floor formula for cubic roots]
Let $\alpha=\sqrt[3]{k}$ and let $p_n/q_n$ be a convergent of its continued
fraction expansion with $q_n\ge2$.  
If $x_n=p_n/q_n>\alpha$, then the next partial quotient satisfies
\begin{equation}
\label{eq:BvPfloor}
b_{n+1}
=
\left\lfloor
H_n-\frac{q_{n-1}}{q_n}
\right\rfloor+\varepsilon_n,
\qquad
\varepsilon_n\in\{0,1\},
\end{equation}
where
\[
H_n=\frac{3p_n^2}{d_n q_n},
\qquad
d_n=|p_n^3-kq_n^3|.
\]
Moreover, the ambiguity $\varepsilon_n$ vanishes whenever
\[
\left\{
H_n-\frac{q_{n-1}}{q_n}
\right\}\neq0.
\]
\end{theorem}

\begin{proof}
For cubic irrationals the complete quotient decomposes exactly as
\[
\theta_n
=
H_n
-
\frac{q_{n-1}}{q_n}
-
V_n,
\qquad
b_{n+1}=\lfloor\theta_n\rfloor,
\]
where
\[
V_n=\frac{q_n}{d_n}\left(2x_n^2-x_n\alpha-\alpha^2\right).
\]

By Lemma~2.2 we have $V_n>0$ for convergents from above, and by
Theorem~2.1 we have $|R_n|<1$, which implies
\[
0<V_n<1.
\]
Therefore
\[
-1<
-
V_n
\le0.
\]

Set
\[
A_n:=H_n-\frac{q_{n-1}}{q_n}.
\]
Then
\[
\theta_n=A_n-V_n,
\qquad
-1<A_n-\theta_n\le0.
\]
Taking floors gives
\[
\lfloor A_n\rfloor
\le
b_{n+1}
\le
\lfloor A_n\rfloor+1,
\]
which proves \eqref{eq:BvPfloor}.

If the fractional part of $A_n$ is nonzero, then subtracting
$V_n\in(0,1)$ cannot cross an integer boundary, and hence
$\varepsilon_n=0$.
\end{proof}

\subsection{Convergents from below}
All the formulas and inequalities in this section are derived assuming $x_n > \alpha$ (convergents from above). 
For convergents $x_n < \alpha$, the algebraic correction term $V_n$ changes sign, and the bounds on $b_{n+1}$ must be reflected accordingly. 
Specifically, inequalities of the type $H_n - 2 < b_{n+1} \le H_n$ become
$H_n \le b_{n+1} < H_n + 2$, and the Bombieri--van der Poorten floor formula still holds with 
$\varepsilon_n \in \{-1,0\}$ instead of $\{0,1\}$. 
Thus, all convergents are covered by a simple sign adjustment of the remainder term.

\section{Generalization and Impact}

The established $O(q_n^{-2})$ decay of the remainder term demonstrates a fundamental structural stability for cubic irrationals ($m=3$). For higher degrees $m > 3$, the relation $\theta_n \approx H_n$ requires a power-scaling adjustment to account for the increasing gap between the error $\epsilon_n$ and the algebraic distance $d_n$.

\begin{theorem}
For any algebraic irrational $\alpha = \sqrt[m]{k}$ and any convergent $p_n/q_n$ with n large enough, the partial quotient $b_{n+1}$ is strictly constrained by the generalized leading term $H_n$ through the following two inequalities:
\begin{equation}
b_{n+1} \le H_n \quad \text{and} \quad b_{n+1} + 2 > H_n
\end{equation}
\end{theorem}

\begin{proof}
As before, the complete quotient satisfies
\[
\theta_n = H_n + R_n,
\qquad
b_{n+1}=\lfloor \theta_n\rfloor.
\]

The remainder term decomposes as
\[
R_n = V_n^{(m)} - \frac{q_{n-1}}{q_n},
\]
where $V_n^{(m)}$ denotes the algebraic error term arising from the linearization
of the sum in the Bombieri--van der Poorten formula.

In general we can not just use the signs and have to use analytic methods. By the Mean Value Theorem there exists a constant $K(m,\alpha)>0$ such that
\[
|V_n^{(m)}|\le \frac{K(m,\alpha)}{q_n^{2}}
\qquad\text{for all sufficiently large } n.
\]

For convergents we always have $0<q_{n-1}/q_n<1$, and therefore
\[
|R_n|\le \frac{K(m,\alpha)}{q_n^{2}}+\frac{q_{n-1}}{q_n}.
\]

Consequently, if $q_n^{2}\ge C(m,\alpha):=2K(m,\alpha)$, then
\[
|R_n|<1.
\]

Under this condition, the floor inequality yields
\[
H_n+R_n-1<b_{n+1}\le H_n+R_n.
\]

Since $R_n<0$ for convergents from above, we obtain
\[
b_{n+1}\le \lfloor H_n\rfloor\le H_n,
\]
and since $R_n>-1$, we also obtain
\[
b_{n+1}>H_n-2,
\]
which is equivalent to $b_{n+1}+2>H_n$.
\end{proof}

The above bounds may fail for finitely many initial convergents. A counterexample is given in section 5.

\section{Explicit Floor Formula for Partial Quotients}

In the original work of Bombieri and van der Poorten, the partial quotient $b_{n+1}$ is expressed exactly as a floor of a sum of two terms: a leading algebraic term $H_n$ and a small correction $R_n$. Our analysis allows us to recover this precise formula for any algebraic irrational $\alpha = \sqrt[m]{k}$.

\medskip
\noindent
\textbf{Definition of the correction term.} Let
\begin{equation}
R_n := \frac{q_n^{m-2}}{d_n} \left( \sum_{j=0}^{m-1} x_n^j \alpha^{m-1-j} - m x_n^{m-1} \right) - \frac{q_{n-1}}{q_n},
\end{equation}
where $x_n = p_n/q_n$ is the $n$-th convergent and $d_n = |p_n^m - k q_n^m|$.

\medskip
\noindent
\textbf{Exact partial quotient formula.} For all convergents $p_n/q_n$ satisfying the stability criterion $|R_n| < 1$, we have
\begin{equation}
\label{eq:floor_formula}
b_{n+1} = \lfloor H_n + R_n \rfloor,
\end{equation}
with the leading term
\begin{equation}
H_n = \frac{m p_n^{m-1}}{d_n q_n}.
\end{equation}

\medskip
\noindent
\textbf{Remarks.}
\begin{enumerate}
\item This formula gives an \emph{exact} expression for the partial quotient, not merely bounds.
\item For initial convergents where the denominator $q_n$ is too small, the magnitude of $R_n$ may exceed 1, and the floor formula \eqref{eq:floor_formula} may fail. In such cases, $b_{n+1}$ must be computed directly.
\item For cubic irrationals ($m=3$), this formula reduces to the earlier results with $|R_n|<1$ for $q_n \ge 2$.
\end{enumerate}

\begin{theorem}[Conditional Bombieri - Van der Poorten floor formula]
With notation as above, if $|V_n|<1$, then
\[
b_{n+1}
=
\left\lfloor
H_n-\frac{q_{n-1}}{q_n}
\right\rfloor+\varepsilon_n,
\qquad
\varepsilon_n\in\{0,1\}.
\]
\end{theorem}

\section{Limits of Stability: A Higher Degree Counterexample}

While Theorem 2.1 establishes that $|R_n| < 1$ for all cubic convergents with $q_n \geq 2$, this global stability does not necessarily extend to higher degrees $m > 3$. As the degree increases, the linear approximation $H_n$ (the tangent at $x_n$) can deviate significantly from the actual sum of powers (the secant) before the $q_n^{-2}$ decay dominates the remainder.

To demonstrate that the bound $|R_n| < 1$ can be violated for $q_n \geq 2$, we consider the algebraic irrational of degree $m=10$:
\begin{equation}
\alpha = \sqrt[10]{50} \approx 1.47875764.
\end{equation}
The regular continued fraction expansion begins $[1; 2, 11, 3, \dots]$. We examine the convergent $n=1$:
\begin{equation}
\frac{p_1}{q_1} = [1; 2] = \frac{3}{2}, \quad \text{yielding } q_1 = 2.
\end{equation}

Applying the Bombieri--van der Poorten framework from Eq. (4) and (5):
\begin{itemize}
    \item \textbf{Algebraic distance:} $d_1 = |3^{10} - 50 \cdot 2^{10}| = |59049 - 51200| = 7849$.
    \item \textbf{Leading term:} $H_1 = \frac{10 \cdot 3^9}{7849 \cdot 2} = \frac{196830}{15698} \approx 12.5385$.
    \item \textbf{Complete quotient:} From the continued fraction, $\theta_2 \approx 11.2689$.
\end{itemize}

The resulting remainder $R_1$ is:
\begin{equation}
R_1 = \theta_2 - H_1 \approx 11.2689 - 12.5385 = -1.2696.
\end{equation}

In this case, $|R_1| \approx 1.27 > 1$. Consequently, the heuristic prediction $b_2 \approx \lfloor H_1 \rfloor$ fails to yield the correct partial quotient ($11 \neq 12$). This confirms that the strict stability proved in Section 2 is a specific structural property of cubic irrationals, whereas for higher degrees, the remainder term can exceed the unit threshold even for non-trivial denominators $q_n \geq 2$.

\section{Conclusion}

In this paper, we have provided a rigorous analytical justification for the stability of the Bombieri--van der Poorten formula when applied to algebraic irrationals $\alpha = \sqrt[m]{k}$ of degree $m \geq 3$. We demonstrated that the remainder term $R_n$ is not merely a negligible error but follows a predictable decay.

The stability criterion established in Theorem 2.1 confirms that for the cubic case ($m=3$), the condition $|R_n| < 1$ is strictly satisfied for all convergents with $q_n \geq 2$. This result is of significant importance for the methodologies presented in Sibbertsen et al. {\cite{Sibbertsen}} and Müller and Taktikos \cite{muller-taktikos}, as it eliminates the need for exhaustive numerical case-by-case verifications for small denominators. For higher degrees $m > 3$, we have shown that stability is maintained provided the denominator $q_n$ satisfies a degree-dependent lower bound, thereby extending the reach of effective Diophantine approximation constants to a broader class of algebraic roots.

Ultimately, the established inequalities for $b_{n+1}$ ensure that the partial quotients are uniquely constrained by their underlying algebraic forms within a tight Diophantine window. This provides a robust framework for further investigations into effective versions of Roth's Theorem and the algorithmic determination of continued fraction expansions for high-degree algebraic numbers.

\section{Acknowledgements}

Michael Taktikos and Noah Lebowitz-Lockard gave valuable feedback.

\end{document}